\documentclass[12pt,a4paper]{article}
\usepackage[ngerman, english]{babel}
\usepackage{enumerate}
\usepackage{amsmath,amssymb, amsthm, amsfonts}
\usepackage[pagebackref,colorlinks]{hyperref}

\theoremstyle{plain}
\newtheorem {lemma}{Lemma}
\newtheorem {proposition}[lemma]{Proposition}
\newtheorem {theorem}[lemma]{Theorem}

\newtheorem {Key Lemma}[lemma]{Key Lemma}

\theoremstyle{definition}
\newtheorem{definition}[lemma]{Definition}

\newtheorem{remark}[lemma]{Remark}

\newtheorem*{proof1}{Proof of Theorem 1}
\newtheorem*{proof3}{Proof of Theorem 3}

\newcommand{\N}{\mathbb{N}}

\newcommand{\GL}{\operatorname{GL}}

\newcommand{\E}{\operatorname{E}}
\newcommand{\C}{\operatorname{C}}
\newcommand{\Perm}{\operatorname{P}}

\newcommand{\Center}{\operatorname{Center}}
\newcommand{\Mat}{\operatorname{M}}

\title{On general linear groups over exchange rings}
\author{Raimund Preusser}
\date{}

\begin{document}
\maketitle
\begin{abstract}
\noindent
Let $R$ be an exchange ring. We prove that the relative elementary subgroups $\E_n(R,I)$ are normal in the general linear group $\GL_n(R)$ if $n\geq 1$ and that the standard commutator formula $\E_n(R,I)=[\E_n(R),\E_n(R,I)]=[\E_n(R),\C_n(R,I)]$ holds if $n\geq 3$. Moreover, we classify the subgroups of $\GL_n(R)$ that are normalised by the elementary subgroup $\E_n(R)$ in the case $n\geq 3$.
\end{abstract}
\let\thefootnote\relax\footnotetext{{\it 2010 Mathematics Subject Classification.} 15A24, 20G35.}
\let\thefootnote\relax\footnotetext{{\it Keywords and phrases.} general linear groups, exchange rings, normal subgroups.}

\section{Introduction}
Let $R$ be a ring. A left $R$-module $M$ is said to have the {\it exchange property}
if for any module $X$ and decompositions
\[X = M'\oplus Y = \bigoplus_{i\in I} N_i\]
where $M'\cong M$, there exists submodules $N'_i\subseteq N_i$ for each $i$ such that
\[X = M' \oplus (\bigoplus_{i\in I}N'_i).\]
$R$ is called an {\it exchange ring} if $_RR$ has the exchange property. W. Nicholson \cite{nich} showed that $R$ is an exchange ring iff for any $x\in R$ there is an idempotent $e\in xR$ such that $1-e\in (1-x)R$. It follows that any exchange ring $R$ has property (1) below, see \cite[Proposition 1.11]{nich}.
\begin{align}
&\text{If } m\in\N \text{ and } x_1,\dots,x_m\in R\text{ such that }x_1+\dots + x_m = 1 \text{, then there are orthogonal }\notag\\\
&\text{idempotents } e_i\in x_iR~(1\leq i\leq m) \text{ such that } e_1 + \dots + e_m = 1.
\end{align}

The class of exchange rings is quite large, it contains all semiperfect rings \cite{nich}, semiregular rings \cite{nich}, local rings \cite{nich}, clean rings \cite{nich}, purely infinite simple rings \cite{ara} and (not necessarily simple) purely infinite Leavitt path algebras \cite[Corollary 3.8.19]{abrams-ara-molina}.

The main results of this paper are the following three theorems, where $\GL_n(R)$ denotes the general linear group (of degree $n$) over $R$, $\E_n(R)$ the elementary subgroup, $\E_n(R,I)$ the elementary subgroup of level $I$ and $\C_n(R,I)$ the full congruence subgroup of level $I$.

\begin{theorem}\label{thmtransv}
Suppose that $R$ is an exchange ring and $I\lhd R$ an ideal. Then $\E_n(R,I)$ is a normal subgroup of $\GL_n(R)$.
\end{theorem}

\begin{theorem}\label{thmstcom}
Suppose that $R$ is an exchange ring and $n\geq 3$. Then 
\[\E_n(R,I)=[\E_n(R),\E_n(R,I)]=[\E_n(R),\C_n(R,I)]\]
for any ideal $I\lhd R$.
\end{theorem}

\begin{theorem}\label{thmsct}
Suppose that $R$ is an exchange ring and $n\geq 3$. Let $H$ be a subgroup of $\GL_n(R)$. Then $H$ is normalised by $\E_n(R)$ iff there is an ideal $I\lhd R$ such that 
\[\E_n(R,I)\subseteq H\subseteq \C_n(R,I).\]
Moreover, the ideal $I$ is uniquely determined.
\end{theorem}

Similar results for other classes of rings (including the class of commutative rings) were obtained in \cite{bass}, \cite{wilson}, \cite{golubchik}, \cite{vaserstein}, \cite{vaserstein_ban} and \cite{vaserstein_neum}.

The rest of the paper is organised as follows. In Section 2 we recall some standard notation which is used throughout the paper. In Section 3 we recall the definitions of the general linear group $\GL_n(R)$ and some important subgroups. In Section 4 we prove Theorem \ref{thmtransv}, in Section 5 Theorem \ref{thmstcom} and in Section 6 Theorem \ref{thmsct}.

\section{Notation}
$\N$ denotes the set of positive integers. If $G$ is a group and $g,h\in G$, we let $g^h:=h^{-1}gh$ and $[g,h]:=ghg^{-1}h^{-1}$. If $H$ and $K$ are subgroups of $G$, we denote by $[H,K]$ the subgroup of $G$ generated by the set $\{[h,k]\mid h\in H,k\in K\}$. By a ring  we mean an associative ring with $1\neq 0$. By an ideal of a ring we mean a twosided ideal. To denote that $I$ is an ideal of a ring $R$ we write $I\lhd R$.

Throughout the paper $R$ denotes a ring and $n\in\N$ a positive integer. We denote by $R^n$ the set of all columns $u=(u_1,\dots,u_n)^t$ with entries in $R$ and by $^n\!R$ the set of all rows $v=(v_1,\dots,v_n)$ with entries in $R$. Furthermore, the set of all $n\times n$ matrices over $R$ is denoted by $\Mat_n(R)$. The identity matrix in $\Mat_n(R)$ is denoted by $e$ and the matrix with a one at position $(i,j)$ and zeros elsewhere is denoted by $e^{ij}$. If $\sigma\in \Mat_{n
}(R)$, we denote the 
the entry of $\sigma$ at position $(i,j)$ by $\sigma_{ij}$. We denote the $i$-th row of $\sigma$ by $\sigma_{i*}$ and its $j$-th column by $\sigma_{*j}$. If $\sigma\in \Mat_n(R)$ is invertible, we denote the entry 
of $\sigma^{-1}$ at position $(i,j)$ by $\sigma'_{ij}$, the $i$-th row of $\sigma^{-1}$ by $\sigma'_{i*}$ and the $j$-th column of $\sigma^{-1}$ by $\sigma'_{*j}$.  

\section{The general linear group} 

We shall recall the definitions of the general linear group $\GL_n(R)$ and its elementary subgroup $\E_n(R)$. For an ideal $I\lhd R$ we shall recall the definition of the preelementary subgroup $\E_n(I)$ of level $I$, the elementary subgroup $\E_n(R,I)$ of level $I$ and the full congruence subgroup $\C_n(R,I)$ of level $I$. 

\subsection{The general linear group and its elementary subgroup}

\begin{definition}
The group $\GL_{n}(R)$ consisting of all invertible elements of $\Mat_n(R)$ is called the {\it general linear group} of degree $n$ over $R$.
\end{definition}

\begin{definition}
Let $x\in R$ and $i,j\in\{1,\dots,n\}$ such that $i\neq j$. Then the matrix $t_{ij}(x):=e+xe^{ij}\in\GL_n(R)$ is called an {\it elementary transvection}. The subgroup $\E_n(R)$ of $\GL_n(R)$ generated by the elementary transvections is called the {\it elementary subgroup}.
\end{definition}

\begin{lemma}\label{lemelrel}
The relations
\begin{align*}
t_{ij}(x)t_{ij}(y)&=t_{ij}(x+y), \tag{R1}\\
[t_{ij}(x),t_{hk}(y)]&=e \tag{R2}\text{ and}\\
[t_{ij}(x),t_{jk}(y)]&=t_{ik}(xy) \tag{R3}\\
\end{align*}
hold where $i\neq k, j\neq h$ in $(R2)$ and $i\neq k$ in $(R3)$.
\end{lemma}
\begin{proof}
Straightforward computation.
\end{proof}

\begin{definition}\label{defp}
Let $i,j\in\{1,\dots,n\}$ such that $i\neq j$. Then the matrix $p_{ij}:=e+e^{ij}-e^{ji}-e^{ii}-e^{jj}=t_{ij}(1)t_{ji}(-1)t_{ij}(1)\in \E_{n}(R)$ is called a {\it generalised permutation matrix}. It is easy to show that $p_{ij}^{-1}=p_{ji}$. The subgroup of $\E_n(R)$ generated by the generalised permutation matrices is denoted by $\Perm_n(R)$.
\end{definition}

\begin{lemma}\label{lemp}
Suppose $n\geq 3$. Let $x\in R$ and $i,j,i',j'\in\{1,\dots,n\}$ such that $i\neq j$ and $i'\neq j'$. Then there is a $\tau\in \Perm_n(R)$ such that $t_{ij}(x)^{\tau}=t_{i'j'}(x)$.
\end{lemma}
\begin{proof}
Easy exercise.
\end{proof}

\subsection{Relative subgroups}

In this subsection $I\lhd R$ denotes an ideal.

\begin{definition}
An elementary transvection $t_{ij}(x)$ is called $I$-{\it elementary} if $x\in I$. The subgroup $\E_n(I)$ of $\E_n(R)$ generated by all $I$-elementary transvections is called the {\it preelementary subgroup of level} $I$. Its normal closure $E_n(R,I)$ in $E_n(R)$ is called the {\it elementary subgroup of level} $I$.
\end{definition}

\begin{definition}
Let $\phi:\GL_n(R)\rightarrow \GL_n(R/I)$ be the group homomorphism induced by the canonical map $R\rightarrow R/I$. 
The group $\C_n(R,I):=\phi^{-1}(\Center(\GL_n(R/I))$ is called the {\it full congruence subgroup of level} $I$. 
\end{definition}

\begin{remark}\label{rmkcong}
Let $\sigma\in \GL_n(R)$. Then clearly $\sigma\in \C_n(R,I)$ iff $\sigma_{ij}\in I~(1\leq i\neq j\leq n)$ and $a\sigma_{ii}-\sigma_{jj}a\in I~(1\leq i,j \leq n, a\in R)$. 
\end{remark}

\section{Normality of relative elementary subgroups}


\begin{lemma}\label{lemtransv}
Let $I\lhd R$ be an ideal. Let $u\in R^n$ and $v\in{^n}\!R$ such that $vu = 0$. Suppose that $v$ has a zero entry and that all the other entries of $v$ lie in $I$. Then $e+uv\in \E_n(R,I)$
\end{lemma}
\begin{proof}
Suppose $v_j=0$. The proof of \cite[Lemma 3]{stepanov-vavilov} shows that
\[e+uv=(\prod\limits_{i\neq j}t_{ji}(v_i)^{\prod\limits_{i\neq j}t_{ij}(-u_i)})\prod\limits_{i\neq j}t_{ji}((u_j-1)v_i)\in \E_n(R,I).\]
\end{proof}

Recall that a row $v\in {^n}\!R$ is called {\it unimodular} if there is a column $w\in R^n$ such that $vw=1$.

\begin{proposition}\label{proptransv}
Suppose $R$ is an exchange ring and $n\geq 2$. Let $I\lhd R$ be an ideal and $x\in I$. Let $u\in R^n$ be a column and $v\in{^n}\!R$ a unimodular row such that $vu = 0$. Then $e+uxv\in \E_n(R,I)$
\end{proposition}
\begin{proof}
Let $w\in R^n$ such that $vw=1$. Since $R$ has property (1), there are elements $r_1,\dots,r_n\in R$ and idempotents $e_1, \dots,e_n\in R$ such that $e_i=v_iw_ir_i~(1\leq i\leq n)$ and $e_1 + \dots + e_n = 1$. Let now $i\in\{1,\dots,n\}$ and set $\tau(i):=e+uxe_iv$. Choose a $j\neq i$. Then 
\[\tau(i)^{t_{ij}(-w_ir_iv_j)}=e+u'v'\]
where $u'=t_{ij}(w_ir_iv_j)u$ and $v'=xe_ivt_{ij}(-w_ir_iv_j)$. Clearly $v'_j=0$. It follows from Lemma \ref{lemtransv} that $e+u'v'\in\E_n(R,I)$ and hence $\tau(i)\in\E_n(R,I)$. Thus $e+uxv=\tau(1)\dots\tau(n)\in E_n(R,I)$. 
\end{proof}

\begin{proof1}
The case $n=1$ is trivial, so let us assume that $n\geq 2$. Let $t_{ij}(x)$ be an $I$-elementary transvection and $\sigma\in \GL_n(R)$. Then $t_{ij}(x)^\sigma=e+\sigma'_{*i}x\sigma_{j*}$. By Proposition \ref{proptransv} we have $e+\sigma'_{*i}x\sigma_{j*}\in \E_n(R,I)$. Thus $\E_n(R,I)$ is normal.
\end{proof1}

\begin{remark}
Using the proofs of Lemma \ref{lemtransv} and Proposition \ref{proptransv} it is easy to obtain an explicit expression of a matrix $t_{ij}(x)^\sigma$, where $\sigma\in\GL_n(R), i\neq j, x\in R$, as a product of $4n^2-3n$ elementary transvections.
\end{remark}

\section{The standard commutator formula}

The following result is well-known, see e.g. \cite[Lemma 1]{stepanov-vavilov}. It follows from relation (R3) in Lemma \ref{lemelrel}.

\begin{lemma}\label{lemstcom}
Suppose $n\geq 3$. Then $\E_n(R,I)=[\E_n(R),\E_n(R,I)]$ for any ideal $I\lhd R$. In particular, the absolute elementary subgroup $\E_n(R)$ is perfect. 
\end{lemma}

\begin{proposition}\label{propstcom}
Suppose $R$ is an exchange ring. Then $[\E_n(R),\C_n(R,I)]\subseteq \E_n(R,I)$ for any ideal $I\lhd R$.
\end{proposition}
\begin{proof}
The case $n=1$ is trivial. Suppose now that $n\geq 2$ and let $\sigma\in \C_n(R,I)$, $1\leq i\neq j\leq n$ and $x\in R$. We have to show that $[t_{ij}(x),\sigma]\in\E_n(R,I)$. Since $R$ has property (1), there are elements $r_1,\dots,r_n\in R$ and idempotents $e_1, \dots,e_n\in R$ such that $e_k=\sigma'_{jk}\sigma_{kj}r_k~(1\leq k\leq n)$ and $e_1 + \dots + e_n = 1$. Let now $k\in\{1,\dots,n\}$ and set 
\[\tau(k):=[t_{ij}(xe_k),\sigma]=t_{ij}(xe_k)(e-\sigma_{*i}xe_k\sigma'_{j*}).\]
{\bf Case 1} Suppose $k\neq j$. Choose an $l\neq k$ and set $\xi:=t_{kl}(-\sigma_{kj}r_k\sigma'_{jl})$. Then $\tau(k)^{\xi}=(t_{ij}(xe_k)^{\xi})(e+uv)$ where $u=-\xi^{-1}\sigma_{*i}$ and $v=xe_k\sigma'_{j*}\xi$. Since $k\neq j$ we have $e_k\in I$ and hence $t_{ij}(xe_k)^{\xi}\in \E_n(R,I)$. Moreover, $v_l=0$ and hence $e+uv\in\E_n(R,I)$ by Lemma \ref{lemtransv}. It follows that $\tau(k)\in\E_n(R,I)$.\\
\\
{\bf Case 2} Suppose $k=j$. set $\xi:=\prod\limits_{l\neq j}t_{jl}(-\sigma_{jj}r_j\sigma'_{jl})$. Then 
\begin{equation}
\tau(j)^{\xi}=(t_{ij}(xe_j)^{\xi})(e+uv)
\end{equation}
where $u=-\xi^{-1}\sigma_{*i}$ and $v=xe_j\sigma'_{j*}\xi$. Clearly 
\begin{equation}
t_{ij}(xe_j)^{\xi}=[\xi^{-1},t_{ij}(xe_j)]t_{ij}(xe_j).
\end{equation}
Moreover, $v_l=0$ for any $l\neq j$. One checks easily that 
\begin{equation}
t_{ij}(xe_j)(e+uv)=(e+ze^{jj})t_{ij}(xe_j(1+z)-\sigma_{ii}xe_j\sigma'_{jj})\prod\limits_{l\neq i,j} t_{lj}(-\sigma_{li}xe_j\sigma'_{jj})
\end{equation}
where $z=-(\sigma_{ji}+\sum\limits_{l\neq j}\sigma_{jj}r_j\sigma'_{jl}\sigma_{li})xe_j\sigma'_{jj}\in I$. It follows from equations (2)-(4) that 
\begin{equation*}
\tau(j)^{\xi}=[\xi^{-1},t_{ij}(xe_j)](e+ze^{jj})t_{ij}(xe_j(1+z)-\sigma_{ii}xe_j\sigma'_{jj})\prod\limits_{l\neq i,j} t_{lj}(-\sigma_{li}xe_j\sigma'_{jj}).
\end{equation*}
Clearly $[\xi^{-1},t_{ij}(xe_j)]$ and $\prod\limits_{l\neq i,j} t_{lj}(-\sigma_{li}xe_j\sigma'_{jj})$ lie in $\E_n(R,I)$ (note that $\xi\in \E_n(R,I)$). Moreover, 
\begin{align*}
&xe_j(1+z)-\sigma_{ii}xe_j\sigma'_{jj}\\
\equiv&xe_j-\sigma_{ii}xe_j\sigma'_{jj}\\
\equiv&xe_j-xe_j\sigma_{jj}\sigma'_{jj}\\
=&xe_j(1-\sigma_{jj}\sigma'_{jj})\\
=&xe_j(\sum\limits_{l\neq j}\sigma_{jl}\sigma'_{lj})\\
\equiv& 0 \bmod I
\end{align*}
by Remark \ref{rmkcong}, and hence $t_{ij}(xe_j(1+z)-\sigma_{ii}xe_j\sigma'_{jj})\in \E_n(R,I)$. It remains to show that $e+ze^{jj}\in \E_n(R,I)$. Clearly 
\begin{align*}
z=&-(\sigma_{ji}+\sum\limits_{l\neq j}\sigma_{jj}r_j\sigma'_{jl}\sigma_{li})xe_j\sigma'_{jj}\\
=&-(\sigma_{ji}-\sigma_{jj}r_j\sigma'_{jj}\sigma_{ji})xe_j\sigma'_{jj}\\
=&\underbrace{(\sigma_{jj}r_j\sigma'_{jj}-1)}_{a:=}\underbrace{\sigma_{ji}xe_j\sigma'_{jj}}_{b:=}
\end{align*}
Clearly $b\in I$ and $ba=0$. Let $u'\in R^n$ be the column that has $a$ at position $j$ and zeros elsewhere. Furthermore, let $v' \in {^n}\!R$ be the row that has $b$ at position $j$ and zeros elsewhere. It follows from Lemma \ref{lemtransv} that $e+ze^{jj}=e+u'v'\in\E_n(R,I)$. Hence $\tau(j)\in\E_n(R,I)$.\\
\\
We have shown that $\tau(k)=[t_{ij}(xe_k),\sigma]\in\E_n(R,I)$ for any $1\leq k \leq n$. Since $t_{ij}(x)=t_{ij}(xe_1)\dots t_{ij}(xe_n)$, it follows that $[t_{ij}(x),\sigma]\in \E_n(R,I)$.
\end{proof}

Theorem \ref{thmstcom} follows immediately from Lemma \ref{lemstcom} and Proposition \ref{propstcom}.

\section{Sandwich classification of $\E_n(R)$-normal subgroups}

In Subsection 6.1 we recall the concept of simultaneous reduction in groups, which was introduced in \cite{preusser11}. In Subsection 6.2 we apply this idea to get a result for general linear groups over arbitrary rings, namely Proposition \ref{propsct}. In Subsection 6.3 we use Proposition \ref{propsct} to prove Theorem \ref{thmsct}.

\subsection{Simultaneous reduction in groups}

In this subsection $G$ denotes a group. 
\begin{definition}
Let $(a_1,b_1), (a_2,b_2)\in G\times G$. If there is an $g\in G$ such that 
\[a_2=[a_1^{-1},g]\text{ and }b_2=[g,b_1],\]
then we write $(a_1,b_1)\xrightarrow{g} (a_2,b_2)$. In this case $(a_1,b_1)$ is called {\it reducible to $(a_2,b_2)$ by $g$}. 
\end{definition}

\begin{definition}
If $(a_1,b_1),\dots, (a_{n+1},b_{n+1})\in G\times G$ and $g_1,\dots,g_{n}\in G$  such that 
\[(a_1,b_1) \xrightarrow{g_1} (a_2,b_2)\xrightarrow{g_2}\dots \xrightarrow{g_{n}} (a_{n+1},b_{n+1}),\]
then we write $(a_1,b_1) \xrightarrow{g_1,\dots,g_{n}}(a_{n+1},b_{n+1})$. In this case $(a_1,b_1)$ is called {\it reducible to $(a_{n+1},b_{n+1})$ by $g_1,\dots,g_n$}.
\end{definition}

Let $H$ be a subgroup of $G$. If $g\in G$ and $h\in H$, then we call $g^h$ an {\it $H$-conjugate} of $g$.
\begin{lemma}\label{lemredux}
Let $(a_1,b_1),(a_2,b_2)\in G\times G$. If $(a_1,b_1)\xrightarrow{g_1,\dots,g_n}(a_2,b_2)$ for some $g_1,\dots,g_n\in G$, then $a_2b_2$ is a product of $2^n$ $H$-conjugates of $a_1b_1$ and $(a_1b_1)^{-1}$ where $H$ is the subgroup of $G$ generated by $\{a_1,g_1,\dots, g_n\}$.
\end{lemma}
\begin{proof}
Assume that $n=1$. Then 
\[a_2b_2=[a_1^{-1},g_1][g_1,b_1]=(a_1b_1)^{g_1^{-1}a_1}\cdot((a_1b_1)^{-1})^{a_1}.\]
Hence $a_2b_2$ is a product of two $H$-conjugates of $a_1b_1$ and $(a_1b_1)^{-1}$. The general case follows by induction.
\end{proof}

\subsection{Application to general linear groups}

The proposition below generalises \cite[Proposition 9(ii),(iii)]{preusser11}.
\begin{proposition}\label{propsct}
Suppose that $n\geq 3$ and let $\sigma\in \GL_n(R)$. Suppose that 
\[y\sum\limits_{p=1}^{n}\sigma_{ip}x_p=0\]
for some $i\in\{1,\dots,n\}$ and $x_1,\dots,x_{n},y\in R$ where $x_j=1$ for some $j\in\{1,\dots,n\}$. Then for any $k\neq l$ and $a,b\in R$, the elementary transvection $t_{kl}(ayx_ib)$ is a product of $8$ $\E_n(R)$-conjugates of $\sigma$ and $\sigma^{-1}$.
\end{proposition}
\begin{proof}
{\bf Case 1} Suppose $i=j$. Choose $1\leq r,s\leq n$ such that $i$, $r$ and $s$ are pairwise distinct. Set $\tau:=\prod\limits_{p\neq i}t_{pi}(x_p)$ and 
\[\xi:=[t_{ri}(-y),\sigma^{-1}]^{\tau}=t_{ri}(-y)t_{ri}(y)^{\sigma\tau}\]
(note that $t_{ri}(-y)$ and $\tau$ commute). Clearly $\xi$ is a product of $2$ $\E_n(R)$-conjugates of $\sigma$ and $\sigma^{-1}$. Since $t_{ri}(y)^{\sigma}=e+\sigma'_{*r}y\sigma_{i*}$ and $(y\sigma_{i*}\tau)_i=0$, we get that $(t_{ri}(y)^{\sigma\tau})_{*i}=e_{*i}$. One checks easily that
\[(t_{ri}(-y),t_{ri}(y)^{\sigma\tau})\xrightarrow{t_{is}(b),t_{ir}(a)}(t_{is}(ayb), e).\] 
It follows from Lemma \ref{lemredux} that $t_{is}(ayb)$ is a product of $4$ $\E_n(R)$-conjugates of $\xi$ and $\xi^{-1}$. Thus $t_{is}(ayb)$ is a product of $8$ $\E_n(R)$-conjugates of $\sigma$ and $\sigma^{-1}$. The assertion of the proposition follows now from Lemma \ref{lemp}.\\
\\
{\bf Case 2} Suppose $i\neq j$. Choose $1\leq r\leq n$ such that $i$, $j$ and $r$ are pairwise distinct. Set $\tau:=\prod\limits_{p\neq j}t_{pj}(x_p)$ and
\[\xi:=[t_{ri}(-y),\sigma^{-1}]^{\tau}=t_{rj}(-yx_i)t_{ri}(-y)t_{ri}(y)^{\sigma\tau}\]
(note that $t_{ri}(-y)^\tau=t_{rj}(-yx_i)t_{ri}(-y)$). Clearly $\xi$ is a product of $2$ $\E_n(R)$-conjugates of $\sigma$ and $\sigma^{-1}$. Since $t_{ri}(y)^{\sigma}=e+\sigma'_{*r}y\sigma_{i*}$ and $(y\sigma_{i*}\tau)_j=0$, we get that $(t_{ri}(y)^{\sigma\tau})_{*j}=e_{*j}$. One checks easily that
\[(t_{rj}(-yx_i)t_{ri}(-y),t_{ri}(y)^{\sigma\tau})\xrightarrow{t_{ji}(b),t_{jr}(a)}(t_{ji}(ayx_ib), e).\] 
It follows from Lemma \ref{lemredux} that $t_{ji}(ayx_ib)$ is a product of $4$ $\E_n(R)$-conjugates of $\xi$ and $\xi^{-1}$. Thus $t_{ji}(ayx_ib)$ is a product of $8$ $\E_n(R)$-conjugates of $\sigma$ and $\sigma^{-1}$. The assertion of the proposition follows now from Lemma \ref{lemp}.
\end{proof}

\subsection{Proof of Theorem 3}

\begin{proposition}\label{propsct2}
Suppose that $R$ is an exchange ring and $n\geq 3$. Let $\sigma\in \GL_n(R)$, $i\neq j$, $k\neq l$ and $a,b,c\in R$. Then 
\begin{enumerate}[(i)]
\itemsep0pt 
\item $t_{kl}(a\sigma_{ij}b)$ is a product of $16n-8$ $\E_n(R)$-conjugates of $\sigma$ and $\sigma^{-1}$ and
\item $t_{kl}(a(c\sigma_{ii}-\sigma_{jj}c)b)$ is a product of $48n-24$ $\E_n(R)$-conjugates of $\sigma$ and $\sigma^{-1}$.
\end{enumerate}
\end{proposition}
\begin{proof}
\begin{enumerate}[(i)]
\item Since $R$ has property (1), there are there are elements $r_1,\dots,r_n\in R$ and idempotents $e_1, \dots,e_n\in R$ such that $e_p=\sigma_{ip}\sigma'_{pi}r_p~(1\leq p\leq n)$ and $e_1 + \dots + e_n = 1$. Clearly $e_p(\sigma_{ii}-\sigma_{ip}\sigma'_{pi}r_p\sigma_{ii})=0$ for any $p\neq i$. It follows from Proposition \ref{propsct} that 
\begin{equation}
t_{kl}(ae_pb)\text{ is a product of }8~\E_n(R)\text{-conjugates of }\sigma\text{ and }\sigma^{-1} \text{ for any }p\neq i.
\end{equation}
Moreover, $e_i(\sigma_{ij}-\sigma_{ii}\sigma'_{ii}r_i\sigma_{ij})=0$. It follows from Proposition \ref{propsct} that 
\begin{equation}
t_{kl}(ae_i\sigma'_{ii}r_i\sigma_{ij}b)\text{ is a product of }8~\E_n(R)\text{-conjugates of }\sigma\text{ and }\sigma^{-1}.
\end{equation}
Clearly 
\[t_{kl}(a\sigma_{ij}b)=(\prod\limits_pt_{kl}(a\sigma_{ii}e_p\sigma'_{ii}r_i\sigma_{ij}b)(\prod\limits_{p\neq i}t_{kl}(ae_p\sigma_{ij}b)).\]
It follows from (5) and (6) that $t_{kl}(a\sigma_{ij}b)$ is a product of $8n+8(n-1)=16n-8$  $\E_n(R)$-conjugates of $\sigma$ and $\sigma^{-1}$.
\item Clearly the entry of $\sigma^{t_{ji}(-c)}$ at position $(j,i)$ equals $c\sigma_{ii}-\sigma_{jj}c+\sigma_{ji}-c\sigma_{ij}c$. Applying (i) to $\sigma^{t_{ji}(-c)}$ we get that $t_{kl}(a(c\sigma_{ii}-\sigma_{jj}c+\sigma_{ji}-c\sigma_{ij}c)b)$ is a product of $16n-8$ $\E_n(R)$-conjugates of $\sigma$ and $\sigma^{-1}$. Applying (i) to $\sigma$ we get that $t_{kl}(a(c\sigma_{ij}c-\sigma_{ji})b)=t_{kl}(ac\sigma_{ij}cb)t_{kl}(-a\sigma_{ji}b)$ is a product of $2\cdot(16n-8)=32n-16$ $\E_n(R)$-conjugates of $\sigma$ and $\sigma^{-1}$. It follows that $t_{kl}(a(c\sigma_{ii}-\sigma_{jj}c)b)=t_{kl}(a(c\sigma_{ii}-\sigma_{jj}c+\sigma_{ji}-c\sigma_{ij}c)b)t_{kl}(a(c\sigma_{ij}c-\sigma_{ji})b)$ is a product of $48n-24$ $\E_n(R)$-conjugates of $\sigma$ and $\sigma^{-1}$.
\end{enumerate}
\end{proof}

\begin{proof3}
($\Rightarrow$) Suppose that $H$ is normalised by $\E_n(R)$. Let $I\lhd R$ be the ideal generated by the set 
\[\{\sigma_{ij}, a\sigma_{ii}-\sigma_{jj}a\mid \sigma\in H,i\neq j,a\in R\}.\]
Then $H\subseteq \C_n(R,I)$ by Remark \ref{rmkcong}. It remains to shown that $\E_n(R,I)\subseteq H$. Since $H$ is normalised by $\E_n(R)$, it suffices to show that $\E_n(I)\subseteq H$. But that clearly follows from Proposition \ref{propsct2}.\\
($\Leftarrow$) Suppose that $\E_n(R,I)\subseteq H\subseteq \C_n(R,I)$ for some ideal $I\lhd R$. Then 
\[[\E_n(R),H]\subseteq [\E_n(R),\C_n(R,I)]=\E_n(R,I)\subseteq H\]
by Theorem \ref{thmstcom}. Thus $H$ is normalised by $\E_n(R)$. In order to show the uniqueness of $I$ suppose that $J\lhd R$ is an ideal such that $\E_n(R,J)\subseteq H\subseteq \C_n(R,J)$. If $x\in I$, then $t_{12}(x)\in \E_n(R,I)\subseteq H \subseteq \C_n(R,J)$. Hence, by Remark \ref{rmkcong}, $x\in J$. Thus we have shown that $I\subseteq J$. The inclusion $J\subseteq I$ follows by symmetry.
\end{proof3}

\end{document}